\documentclass{arXiv}
\usepackage{fancybox,amsmath,amssymb,color}

\newtheorem{theorem}{Theorem}[section] 
\newtheorem{lemma}[theorem]{Lemma}     
\newtheorem{cor}[theorem]{Corollary}
\newtheorem{prop}[theorem]{Proposition}
\newtheorem{defn}[theorem]{Definition}
\newtheorem{rem}[theorem]{Remark}

\newnumbered{hypothesis}{Hypothesis}
\newnumbered{note}{Note}
\newnumbered{observation}{Observation}
\newnumbered{problem}{Problem}
\newnumbered{question}{Question}
\newnumbered{algorithm}{Algorithm}
\newnumbered{example}{Example}
\newunnumbered{notation}{Notation} 

\numberwithin{equation}{section}
\numberwithin{theorem}{section}

\newcommand{\lij}{\sqrt{\lambda_i+\lambda_j}}
\newcommand{\LIJ}{\lambda_i+\lambda_j}
\newcommand{\IK}{\lambda_i+\lambda_k}
\newcommand{\JK}{\lambda_j+\lambda_k}
\newcommand{\IJK}{(\LIJ)(\lambda_j+\lambda_k)(\lambda_k+\lambda_i)}
\newcommand{\ld}{\sqrt{\lambda_1+\lambda_d}}
\newcommand{\X}{\mathcal{X}}
\newcommand{\Y}{\mathcal{Y}}
\newcommand{\Z}{\mathcal{Z}}
\newcommand{\N}{\mathcal{N}}
\newcommand{\mm}{\mathcal{M}}
\newcommand{\W}{\mathcal{W}}
\newcommand{\px}{\mathcal{P}_{2}(X)}
\newcommand{\ppm}{\mathcal{P}_{2}(M)}
\newcommand{\pr}{\mathcal{P}_2(\R^d)}
\newcommand{\ppac}{\mathcal{P}^{\mathrm{ac}}_{2}(\R^d)}
\newcommand{\pmac}{\mathcal{P}^{\mathrm{ac}}_{2}(M)}
\newcommand{\Sym}{\mathrm{Sym}(d,\R)}
\newcommand{\SSym}{\mathrm{Sym}(2,\R)}
\newcommand{\sym}{\mathrm{Sym}^+(d,\R)}

\newcommand{\diag}{\mathrm{diag}}
\newcommand{\tr}{\mathrm{tr}}
\newcommand{\ten}{{}^{\mbox{\tiny T\hspace{-1.4pt}}}}
\newcommand{\R}{\mathbb{R}}
\newcommand{\id}{\mathrm{id}}
\newcommand{\co}{\cos \theta}

\newcommand{\coo}{\cos \theta_0}

\newcommand{\eep}{\epsilon\mathrm{xp}}

\newcommand{\ccag}{c_{\alpha\gamma}(r,\theta_0)}

\newcommand{\crg}{c_{\alpha\gamma}(r,\theta)}
\newcommand{\cij}{c_{i j}(r,\theta)}
\newcommand{\ccij}{c_{i j}(r,\theta_0)}
\newcommand{\sij}{s_{i j}(r,\theta)}
\newcommand{\ssij}{s_{i j}(r,\theta_0)}
\newcommand{\sik}{s_{i k}(r,\theta)}
\newcommand{\ssik}{s_{i k}(r,\theta_0)}

\title[]{On Wasserstein geometry of the space of Gaussian measures}
\author[]{Asuka TAKATSU}
\date{}

\classno{60D05(primary), 28A33}

\extraline{This work was partially supported 
by Research Fellowships of 
the Japan Society for the Promotion of
Science for Young Scientists.}
\usepackage{enumerate}
\usepackage{graphicx,color}

\begin{document}
\maketitle

\begin{abstract}
The space of Gaussian measures 
on a Euclidean space is geodesically convex 
in the $L^2$-Wasserstein space.
This space is a finite dimensional manifold 
since Gaussian measures are parameterized 
by means and covariance matrices. 
By restricting to the space of Gaussian measures 
inside the $L^2$-Wasserstein space, 
we manage to provide detailed descriptions of
the $L^2$-Wasserstein geometry 
from a Riemannian geometric viewpoint.
We first construct a Riemannian metric 
which induces the $L^2$-Wasserstein distance.
Then we obtain a formula 
for the sectional curvatures of the space of Gaussian measures, 
which is written out in terms of the eigenvalues of the covariance matrix.
\end{abstract}
\section{Introduction}
In this paper, 
we give a formula for sectional curvatures of 
the space of Gaussian measures on $\R^d$ 
with the $L^2$-Wasserstein metric.
Let $N(m,V)$ be the Gaussian measure with mean $m$ 
and covariance matrix $V$.
Namely  $m$ is a vector in $\R^d$
and $V$ is a symmetric positive definite matrix of size $d$ and 
its Radon-Nikodym derivative is given by 
\[
 \frac{d N(m,V)}{d x}= 
        \frac{1}{\sqrt{\det (2\pi V)}}
        \exp\left[ -\frac{1}{2} 
            \langle x -m,V^{-1}(x-m) \rangle \right].
\]
We denote by $\N^d$ the space of Gaussian measures on $\R^d$.
Since Gaussian measures depend only on 
the mean $m$ and the covariance matrix $V$,
the space $\N^d$ is identified with $\R^d \times \sym$, 
where $\sym$ is the space of symmetric positive definite matrices
of size $d$.

The $L^2$-Wasserstein space is the subspace of probability measures 
equipped with a certain distance.
Let $\ppac$ be the set of absolutely continuous probability measures 
with finite second moments on $\R^d$.
Then $\ppac$ is a geodesic space 
and 
all geodesics are given by push-forward measure.
In view of these facts, 
Otto \cite{Ot} regarded 
$\ppac$ as an infinite dimensional formal Riemannian manifold  
and analyzed  the porous medium equations 
as gradient flows on $\ppac$. 

A foundation for this framework was carefully laid out 
by Carrillo-McCann-Villani~\cite{CMV}.
They introduced the new space, \textit{Riemannian length space}.
In short, this space is a length space 
which has an exponential map defined on some tangent vector space 
with a metric.
They proved that $\ppac$ is a Riemannian length space and 
its metric induces the $L^2$-Wasserstein distance.
We call this metric the $L^2$-Wasserstein metric.

McCann showed in \cite{Mc} that 
varying the mean is equivalent to 
a Euclidean translation and 
$\N^d_0$ is a geodesically convex subspace 
of $\ppac$.
When we consider the $L^2$-Wasserstein geometry on $\N^d$,
it suffices to consider the geometry on covariance matrix variations.
We use $\N^d_0$ for the set 
of all Gaussian measures with mean $0$.
We denote by $N(V)$ 
the Gaussian measure with mean $0$, 
and covariance matrix $V$.

In the Riemannian length space, 
if a geodesic from $N(V)$ 
with direction $\psi$ 
passes thorough $N(U)$,
then the gradient of $\psi$ 
is given as a linear map associated 
with a symmetric matrix depending only on $V,U$.
Thus the tangent space at each point can be regarded as 
the space of symmetric matrices $\mathrm{Sym}(d,\R)$.
This identification coincides with 
the viewpoint from the differential structure;
since $\N^d_0$ is identified with $\sym$, 
which is an open subset of 
$\mathrm{Sym}(d,\R)$,
we can consider the tangent space of $\N^d_0$ 
as  $\mathrm{Sym}(d,\R)$.
These observations enable us to 
obtain a formula for the sectional curvature 
of $\N^d_0$ 
by restricting to a geodesically  convex submanifold $\N^d_0$ 
of $\ppac$.
%
\begin{theorem}\label{curv}
For an orthogonal matrix $P$ and 
positive numbers $\{\lambda_i\}_{i=1}^d$, 
we set  $V= P\diag[\lambda_1,\ldots,\lambda_d] \ten P$,
where  $\ten P$ is the transpose matrix of $P$.
Then we can consider the tangent space to $\N^d_0$ 
at $N(V)$ spanned by 
\[
 \left\{
  e_+= \frac{P\left( E_{11}+E_{dd} \right) \ten P}{\ld},\ 
  e_{i j}=\frac{P\left( E_{ii}-E_{j j} \right) \ten P}{\lij},\ 
  f_{i j}=\frac{  P\left( E_{i j}+E_{j i} \right) \ten P}{\lij}   
 \right\}_{1\leq i<j \leq d},
\] 
where $E_{i j}$ is an $(i,j)$-matrix unit, 
whose $(i,j)$-component is $1$, $0$ elsewhere.
Then we obtain the following expressions of the 
sectional curvatures with respect to the vectors:
\begin{align}
 \tag{1}
 &K(e_+,e_{i j})=0 \\ 
 \tag{2}
 &K(e_+,f_{1d})=0 \\
 \tag{3}
 &K(e_+,f_{i j})=
         \frac{3\lambda_i \lambda_j}
            {(\lambda_i+\lambda_j)^2(\lambda_1+\lambda_d)}
          &&(i=1 \text{\ or \ } j=d) \\
 \tag{4}
 &K(e_+,f_{kl})=0 &&(1<k<l<d)\\
 \tag{5}
 &K(e_{i j},e_{kl})=0 \\
 \tag{6}
 &K(e_{i j},f_{kl})=0 &&(\{i,j\}\cap\{k,l\}=\emptyset)\\
 \tag{7}
 &K(e_{i k},f_{i j})=
       \frac{3\lambda_i \lambda_j}
            {(\lambda_i+\lambda_j)^2(\lambda_i+\lambda_k)}
       &&(j \neq k)\\
\tag{8}
 &K(e_{i j},f_{i j})=
       \frac{12\lambda_i\lambda_j}
            {(\lambda_i+\lambda_j)^3}\\
\tag{9}
 &K(f_{i j},f_{kl})=0  &&(\{i,j\}\cap\{k,l\}=\emptyset)\\
\tag{10}
 &K(f_{i j},f_{i k})=
        \frac{3\lambda_j\lambda_k}
             {(\lambda_i+\lambda_j)(\lambda_j+\lambda_k)(\lambda_k+\lambda_i)} 
         &&(j \neq k).
\end{align}    
\end{theorem}
The formula coincides with a formal 
expressions of sectional curvatures of 
$\ppac$ given by Otto~\cite{Ot}.
This formula shows that the sectional curvature of $\N^d_0$ 
is non-negative and
is written out only
in terms of the eigenvalues of the covariance matrix.

The organization of the paper is as follows. 
We start with a review of 
the $L^2$-Wasserstein geometry and 
the Riemannian length space in Section~2.
Then we prove Theorem~\ref{curv} in Section~3,
using the approximate expression of sectional curvature.
We demonstrate the correspondence 
between our results 
and previously obtained result in Section~4. 

\section*{Acknowledgements}
The author would like to express 
her gratitude to Professor Sumio Yamada, 
for his advice, support and encouragement.
His valuable suggestions were essential to the completion of this paper.
\section{Preliminaries}

\subsection{$L^2$-Wasserstein space}
We first review $L^2$-Wasserstein spaces 
(see \cite{Vi}.) Given a complete metric space $(X,d)$,
we denote by $\px$ 
the set of probability measures with finite second moments on $X$. 
\begin{defn}
For $\mu$, $\nu \in \px$, 
a transport plan $\pi$ between $\mu$ and $\nu$
 is a Borel probability measure on $X \times X$ with marginals $\mu$ and $\nu$,
 that is, 
\[
\pi [ A \times X ]=\mu [A],
\quad \pi [ X \times A ]= \nu [A]
\quad \text{for all Borel sets A in $X$.}
\]
Let $\Pi(\mu,\nu)$ be the set of 
transport plans between $\mu$ and $\nu$,  
then 
the $L^2$-Wasserstein distance between $\mu$ and $\nu$
is defined by
\[
W_2(\mu ,\nu ) ^2 = \inf _{\pi \in \Pi (\mu, \nu)}
               \int _{X \times X }d(x,y)^2 d \pi (x,y).
\]
\end{defn}
The $L^2$-Wasserstein distance actually becomes a distance.
We call the pair ($\px,W_2$) the $L^2$-Wasserstein space over $X$.
A transport plan which achieves the infimum 
is called optimal.
Optimal transport plans on Euclidean spaces are characterized 
by the following properties.
\begin{theorem} [(\cite{bre},\cite{MG})]\label{chara} 
Let $\mu$ and $\nu$ be Borel probability measures on $\R^d$. 
If $\mu$ is absolutely continuous 
with respect to Lebesgue measure, then
\begin{enumerate}
\item
there exists a convex function $\psi$ on $\R^d$  
whose gradient $\nabla \psi$ pushes $\mu$ forward to $\nu$.
\item
 this gradient is uniquely determined ($\mu$-almost everywhere.)

\item
 the joint measure $\pi=( \id \times \nabla \psi)_{\sharp} \mu$ is optimal.

\item
  $\pi$ is the only optimal measure in $\Pi(\mu,\nu)$
 unless $W_2(\mu,\nu) = +\infty$.
\end{enumerate} 
Here the push forward measure of $\mu$ 
through measurable map $f:\R^d \to \R^d$, 
denoted by $f_{\sharp} \mu$, is 
defined by $f_{\sharp} \mu[A]=\mu[ f^{-1}(A)]$ 
for all Borel sets $A$ in $\R^d$. 
\end{theorem}
McCann~\cite{Mc} obtained 
the optimal transport plans 
between Gaussian measures on $\R^d$ and 
showed that
the displacement interpolation 
between any two Gaussian measures is also a Gaussian measure.
Namely, $\N^d$ is a geodesically convex subset of
the $L^2$-Wasserstein space.
\begin{lemma} [({\cite[Example~1.7]{Mc}})] \label{push}
For $X \in \sym$, 
we define a symmetric positive definite matrix 
$X^{1/2}$  so that $X^{1/2}\cdot X^{1/2}=X$.
For $N(m,V)$ and $N(n,U)$,
define a symmetric positive definite matrix
\[
W=(w_{i j})=U^{\frac12}(U^{\frac12}V U^{\frac12})^{-\frac12}U^{\frac12}
\]
and the related function
\[
\W (x) =\frac{1}{2} \langle x-m,W(x-m)\rangle 
+ \langle x,n \rangle.
\]
We denote the gradient of $\W$ by $\nabla \W$.
Then, $(\id,\nabla\W)_{\sharp}N(m,V)$ is 
the optimal transport between 
$N(m,V)$ and $N(n,U)$.
If we moreover set 
\[
 l(t)=(1-t)m+tn, \quad 
 W(t)=((1-t)E+t W)V((1-t)E+t W),
\]
then $\{N(l(t),W(t))\}_{t \in[0,1]}$ is a geodesic 
from $N(m,V)$ to $N(n,U)$. 
\end{lemma}
Lemma~\ref{push} enables us to 
obtain the $L^2$-Wasserstein distance on $\N^d_0$.
\begin{theorem}
[(\cite{DL}, \cite{GS}, \cite{KS}, \cite{OP})] \label{kyori}
For $N(m,V)$ and $N(n,U)$, we get
\[
 W_2(N(m,V),N(n,U))^2
=|m-n|^2 +\tr V + \tr U-
 2\tr\left({U^{\frac12}V U^{\frac12}}\right)^{\frac12}.
\]
\end{theorem}
We call $\W$ above a (unique) linear transform 
between $N(m,V)$ and $N(n,U)$.
Let $O(d)$ be the set of orthogonal 
matrices of size $d$.
For $P \in O(d)$,
we denote by $\N^d(P)$ 
the subset of $\N^d$ 
whose covariance matrices are diagonalized by $P$.
\begin{cor} \label{flat}
For any $P \in O(d)$,
$(\N^d(P),W_2)$ is isometric to $\R^d \times (\R_{>0})^d$.
\end{cor}
\begin{proof}
For $N(m,V),N(n,U) \in \N^d(P)$, 
there uniquely exist 
$\{\lambda_i\}_{i=1}^d,\{\sigma_i\}_{i=1}^d \subset \R_{>0}$ 
such that 
\[
 V=P \diag [\lambda_1,\ldots,\lambda_d] \ten{P}, \quad
 U=P \diag [\sigma_1,\ldots,\sigma_d] \ten P.
\] 
By Theorem~\ref{kyori}, we have 
\[
  W_2(N(m,V),N(n,U))^2
 =|m-n|^2 +\sum_{i=1}^d (\lambda_i-\sigma_i)^2.
\]
Therefore a map identifying 
$N(m,V)$ with $(m,(\lambda_1,\ldots,\lambda_d))$ 
is an isometry from 
$\N^d(P)$ to $\R^d \times (\R_{>0})^d$. 
\end{proof}

\begin{rem}
In the $L^2$-Wasserstein geometry,
$\N^1$ is isometric to a Euclidean upper half plane.
While in the Fisher geometry,
$\N^1$ is isometric to a hyperbolic plane 
with constant sectional curvature $-1/2$~{\upshape (see~\cite{amari}.)}
\end{rem}

\subsection{Riemannian length space}
Next, we give $\N^d$ an $L^2$-Wasserstein metric. 
See \cite{CMV} for more detail.
\begin{defn}\label{rls}
Let $\langle \cdot,\cdot \rangle _y$ and $|\cdot|_y$ 
denote an inner product and a norm on a vector space $\mathcal{H}_y$.
A subset $M$ of a length space $(N, \mathrm{dist})$ is called Riemannian
if each $x \in M$ is associated with a map $\eep_x :\mathcal{H}_x \to N$
defined on some inner product space $\mathcal{H}_x$ which gives a surjection
from a star-shaped subset  $\mathcal{K}_x \subset \mathcal{H}_x$ onto $M$
such that the curve $x_s= \eep_x (s p)$ defines an (affinely parameterized)
minimizing geodesic $[0,1] \ni s \mapsto x_s$ linking $x=x_0$ to $y=x_1$
for each $p \in \mathcal{K}_x$.
We moreover assume that there exists $q \in \mathcal{K}_y$
such that $x_s=\eep_y(1-s)q$ and
\[
\mathrm{dist}(\eep_x u,\eep_y v)^2 \leq
\mathrm{dist}(x,y)^2 -2 \langle v, q \rangle _y -2 \langle u, p \rangle _x
+ o (\sqrt{|u|_x^2+|v|_y^2}),
\]
for all $u \in \mathcal{H}_x$ and $v \in \mathcal{H}_y$
as $|u|_x +|v|_y \to 0$.
Dependence of these structures on the base points $x$ and $y$
may be suppressed when it can be inferred from the context.
\end{defn}
%
It was shown in~\cite[Proposition 4.1]{CMV} 
that $\ppac$ 
forms a Riemannian length space with the following methods.

Take $(N,\mathrm{dist})=(\pr,W_2)$ as 
our complete length space and the subset $M=\ppac$.
Fix $\rho \in M$.
Let $\mathrm{spt}(\rho )$ denote smallest closed subset of $\R^d$ 
containing the full mass of $\rho$, 
and let $\Omega _\rho \subset \R^d$ 
denote the interior of the convex hull of $\mathrm{spt}(\rho)$.
We take $\mathcal{H}_\rho = 
\mathcal{H}^{1,2}(\R^d , d \rho) 
\subset C^{0,1}_{\mathrm{loc}}(\Omega _\rho)$ 
to consist of those locally Lipschitz 
continuous functions on $\Omega _\rho$ 
whose first derivatives lie in the weighted space 
$L^2(\R ^d , d \rho;\R ^d)$, 
modulo equivalence with respect to semi-norm
\[
\langle \psi,\psi \rangle_\rho
=\int _{\Omega _\rho} |\nabla \psi (x)| ^2 d \rho(x).
\]
And the exponential map is defined by
\[
\eep _\rho s\psi =[\id + s \nabla \psi] _{\sharp} \rho .
\]
 
Furthermore, they remarked if 
$M^{\prime} \subset M$ is geodesically convex, 
meaning any geodesic lies in $M^{\prime}$ 
whenever its endpoints do, then $M^{\prime}$ is itself a
Riemannian length space with the same tangent space 
and the exponential map as those of $M$, 
but the star-shaped subset is given by
\[
\mathcal{K}_x^{\prime}
=\{ p \in \mathcal{K}_x \bigm | \exp_x p \in M^ {\prime} \}.
\] 
Lemma~\ref{push} implies 
that $\N^d_0$ is a geodesically convex subset of $\ppac$.
Therefore $\N^d_0$ should also be a Riemannian length space,
especially a Riemannian manifold.
If $\W$ is the linear transform between $N(V)$ and $N(U)$,
then $\eep _{N(V)} s \psi =N(U)$ 
if and only if 
$\nabla \psi$ is same as 
$\nabla (\W- |\cdot|^2/2)$.
Identifying linear transforms as their coefficients,
we treat the tangent vector space at each point as $\Sym$. 
Namely,
we can identify the tangent space at $N(V)$ to $\N^d_0$ 
with $\Sym$ based on the idea   
\[
 \exp_{N(V)} t X =N(U(t)), 
\text{where }
U(t)=((1-t)E+t X)V((1-t)E+t X).
\]
Moreover, the inner product becomes 
a Riemannian metric $g$ on $\N^d_0$, 
whose Riemannian distance 
coincides with the $L^2$-Wasserstein distance.
Its expression of $g$ is given by 
\[
 g_{N(V)}(t X,t X)
=\int _{\R^d} |t X x| ^2 d N(V)(x)
=t^2\tr X V X.
\]

\begin{theorem} \label{mani}
Let $\N^d_0$ be the space of Gaussian measures 
with mean $0$ over $\R^d$. 
Then, $\N^d_0$ becomes a $C ^\infty$-Riemannian manifold 
of dimension $d(d+1)/2$ 
and there exists the $L^2$-Wasserstein metric $g$.
If we identify the tangent space at $N(V)$ to $\N^d_0$ 
with $\Sym$ by  
\[
 \exp_{N(V)} t X =N(U(t)), 
\text{where }
U(t)=((1-t)E+t X)V((1-t)E+t X),
\]
then it is explicitly given by 
\[
 g_{N(V)}(X,Y)= \tr X V Y.
\]
\end{theorem}
Theorem~\ref{mani} shows that 
$\{e_+, e_{i j},f_{i j}\}_{1\leq i<j \leq d}$ 
is a set of 
normal vectors.
\section{Proof of Theorem \ref{curv}}
In order to calculate the sectional curvatures 
of $(\N^d_0,g)$, where $g$ is the $L^2$-Wasserstein metric,
we need some lemmas.
%
\begin{lemma} [({\cite[Theorem 3.68]{GHL}})] \label{taylor}
Let $(M,g)$ be a Riemannian manifold.
For any $p \in M$, 
$\{u,v\}$ is an orthonormal basis of a $2$-plane in the tangent space at $p$.
Let
\[
C_r(\theta  )=
\exp _p r (u \cos \theta  + v \sin \theta  ),
\]
and $L(r)$ be the length of the curve $C_r$. 
Then the function $L(r)$ admits an asymptotic expansion
\[
L(r)= 2 \pi r 
\left(
1-\frac{K(u,v)}{6}r^2 +o(r^2)
\right) , \quad as \ \  r \searrow 0,
\]
where $K(u,v)$ is the sectional curvature of the $2$-plane 
spanned by $\{u,v\}$.
\end{lemma}
\begin{lemma}\label{linear}
For $A,B \in 
    \left\{e_+, e_{i j}, f_{i j} \right\}_{1\leq i<j \leq d}$,
 $0 \leq r \ll  1$ and $\theta  \in [0,2\pi]$,
\[
 C_r(\theta )=\exp_{N(V)} r(\cos \theta \cdot A +\sin \theta  \cdot B)
\] 
is a Gaussian measure whose 
covariance matrix 
$X=X(r,\theta )
  =(x_{\alpha\beta})
$
 is given by 
\begin{equation}\label{cov}
 X=[E+ r(\cos \theta \cdot A +\sin \theta  \cdot B)]
     \cdot V \cdot
     [E+r(\cos \theta \cdot A +\sin \theta  \cdot B)],
\end{equation}
where $E$ is the identity matrix.  
\end{lemma}
\begin{proof}
It is clear by Lemma~\ref{push}.
\end{proof}
\begin{flushleft}
\textbf{Proof of (1) and (5)}  
\end{flushleft}
If we choose $A=e_+, B=e_{i j}$ or 
$A=e_{i j}, B=e_{k l}$
in \eqref{cov}, 
then $N(X)$ belongs to $\N^d_0(P)$.
Since  $\N^d_0(P)$ is a flat manifold 
by Corollary~\ref{flat},
the curvatures vanish.

A strategy for proving the remaining case is as follows. 
We first calculate
\[
 W(\theta _0,\theta )=W_2(C_r(\theta _0),C_r(\theta ))^2
\text{\quad and \quad}
 W(\theta _0)=
 \lim_{\theta \rightarrow \theta_0}
 \frac{ W(\theta _0,\theta )}{\theta ^2}.
\]
Then we get
\[
 L(r)=\int_{0}^{2\pi} W(\theta )^{\frac12} d \theta .
\]
Finally we use Lemma~\ref{taylor} 
to obtain the expression of the 
sectional curvatures.
Without loss of generality, we may assume $P=E$.
That is to say, 
\[
e_+=\frac{E_{11}+E_{dd}}{\ld},\ 
e_{i j}=\frac{E_{ii}-E_{j j}}{\lij},\  
f_{i j}=\frac{E_{i j}+E_{j i}}{\lij},\ 
 \text{and }
 V=\diag[\lambda_1,\ldots,\lambda_d],
\]
because we have
\[
   W(\theta _0,\theta )
  =\tr X(r,\theta _0) + \tr X(r,\theta ) -
   2 \tr \left(X(r,\theta _0)^{\frac12}
               X(r,\theta )
               X(r,\theta _0)^{\frac12}\right)^{\frac12}
\]
and the value is invariant under taking conjugation with 
any orthogonal matrix $P$.

For a general symmetric positive definite matrix $X$, 
it is hard to get a concrete expression of $X^{1/2}$.
But if the matrix is size of $2\times 2$, 
the next lemma enables us 
to obtain the value of the trace and the determinant of $X^{1/2}$.
\begin{lemma} \label{square}
Let $M \in \SSym$, then
\[
(\tr M)^2
=\tr M^2 + 2 \det M.
\]
\end{lemma}
\begin{proof}
Setting 
\[
M=
\begin{pmatrix}
a & c \\
c & b \\
\end{pmatrix}, 
\text{ we obtain }
M^2=
\begin{pmatrix}
a^2 +c^2 &c( a+b ) \\
c( a+b )  &b^2+ c^2 \\
\end{pmatrix}.
\]
Therefore, we get
\[
\tr M^2 +2\det M
=a^2+b^2+2c^2-2(a b -c^2)
=(a+b)^2
=(\tr M)^2.
\]
\end{proof}

We set 
\[
 c_{i j}(r,\theta)=\frac{r \cos \theta}{\lij}, \  
 s_{i j}(r,\theta)=\frac{r \sin \theta}{\lij} 
\]
for $1\leq i,j \leq d$, $\theta \in [0,2\pi]$ 
and sufficiently small $r\geq0$. 
\begin{flushleft}
\textbf{Proof of (2) and (8)}  
\end{flushleft}
For (2),  
we take  
$A=f_{1 d}$, $B=e_{+}$ and $I=\{1,d\}$, 
whereas,  for (8),  
take $A=f_{i j}$, $B=e_{i j}$ and $I=\{i,j\}$.
Then we notice that for any $\alpha,\beta \notin I$,
$(\alpha,\beta)$-components of $X$ 
are independent of the variables $r$ and $\theta$.
If we set
\begin{align*}
 &\widetilde{X}(\theta)=
  \begin{pmatrix}
     x_{\alpha\alpha} &x_{\alpha \beta} \\
     x_{\beta \alpha} &x_{\beta \beta}
  \end{pmatrix}
\end{align*}
for $\{\alpha,\beta\}=I$, we obtain 
\begin{align}\label{saisyo}
 W(\theta_0,\theta)
& =\tr \widetilde{X}(\theta_0)+\tr \widetilde{X}(\theta)
   -2\tr\left( \widetilde{X}(\theta_0)^{\frac{1}{2}}
             \widetilde{X}(\theta)
             \widetilde{X}(\theta_0)^{\frac{1}{2}} \right)^{\frac12}.
\end{align}
For (2),
using Lemma~\ref{square},  
we conclude 
\begin{align*}
 W(\theta_0, \theta ) 
=4r^2\sin^2(\theta-\theta_0)\  
\text{ and } 
 \lim_{\theta\rightarrow \theta_0}
 \frac{W(\theta_0,\theta)}{(\theta - \theta_0)^2}=r^2.
\end{align*}
It follows that $L(r)=2\pi r$, proving
 $K(e_{+},f_{1 d})=0$.

For (8), 
in a similar way, we have
\begin{align*}
 W(\theta_0, \theta ) 
=4 r^2 \sin^2 \frac12 (\theta-\theta_0) 
-\frac{4 r^4\lambda_i \lambda_j  \sin^2(\theta-\theta_0)}
{(\lambda_i +\lambda_j)^2a_r(\theta_0,\theta)}+o(|\theta-\theta_0|^2),    
\end{align*}
where
\begin{align*}
a_r(\theta_0,\theta )
 =&\lambda_i
  \left[ (1+\ccij)(1+\cij)+  \ssij \sij \right]\\
  &+\lambda_j
  \left[(1-\ccij)(1-\cij)+ \ssij  \sij \right].
\end{align*}
Since the limit of $a_r(\theta_0,\theta )$ exists 
as $\theta\rightarrow \theta_0$ and
\[
a_r(\theta_0,\theta_0 )
 =(\lambda_i+\lambda_j)(1+r^2)+2(\lambda_i-\lambda_j)r\coo,
\]
we have 
\[
 \lim_{\theta\rightarrow \theta_0}
 \frac{W(\theta_0,\theta)}{(\theta - \theta_0)^2}
=r^2-\frac{4r^4\lambda_i\lambda_j }
  {(\lambda_i+\lambda_j)^2a_r(\theta_0,\theta_0)}.
\]
It follows that
\begin{align*}
  L(r)
&=\int_0^{2\pi} 
 r\left(1-\frac{4r^2\lambda_i\lambda_j }
 {(\lambda_i+\lambda_j)^2a_r(\theta,\theta)}\right)^{\frac12} d\theta \\
&=\int_0^{2\pi} 
 r\left(1-\frac12 \frac{4r^2\lambda_i\lambda_j }
 {(\lambda_i+\lambda_j)^2a_r(\theta,\theta)} +o(r^2)\right) d\theta . 
\end{align*}
Because $a_0(\theta,\theta)=\lambda_i+\lambda_j$, 
using Lemma~\ref{taylor} and
 the bounded convergence theorem, 
we obtain 
\[
 K(e_{i j},f_{i j})
=\frac{12 \lambda_i\lambda_j}{(\lambda_i+\lambda_j)^3}.
\]

\begin{flushleft}
\textbf{Proof of (3) and (7)}  
\end{flushleft}
For (3), 
assuming $i=1$, 
take  
$A=e_+$, $B=f_{1j}$ and $I=\{1,j,d\}$, 
whereas,  for (7), 
assuming $j<k$, 
take $A=e_{i k}$, $B=f_{i j}$ and $I=\{i,j,k\}$.
Since for any $\alpha,\beta \notin I$,
$(\alpha,\beta)$-components of $X$ 
are independent of the variables $r$ and $\theta$,
we obtain 
\begin{align*}
& W(\theta_0,\theta)\\
& =\tr \widetilde{X}(\theta_0)+\tr \widetilde{X}(\theta)
   -2\tr\left(\widetilde{X}(\theta_0)^{\frac{1}{2}}
             \widetilde{X}(\theta)
            \widetilde{X}(\theta_0)^{\frac{1}{2}}\right)^{\frac12} \\
& =\tr \widetilde{Y}(\theta_0)+\tr \widetilde{Y}(\theta)
   -2\tr\left(\widetilde{Y}(\theta_0)^{\frac{1}{2}}
             \widetilde{Y}(\theta)
            \widetilde{Y}(\theta_0)^{\frac{1}{2}}\right)^{\frac12}
 +\frac{r^2\lambda_\gamma}{\lambda_\alpha+\lambda_\gamma}(\co-\coo )^2,
\end{align*}
where 
\[
\widetilde{X}(\theta)
 =\begin{pmatrix}
   x_{\alpha\alpha}&x_{\alpha\beta}&x_{\alpha\gamma} \\
   x_{\beta\alpha}&x_{\beta\beta}&x_{\beta\gamma} \\
   x_{\gamma\alpha}&x_{\gamma\beta}&x_{\gamma\gamma} \\
  \end{pmatrix}
=\begin{pmatrix}
   \widetilde{Y}(\theta)& \ten\  \mathbf{0} \\
   \mathbf{0}& \lambda_{\gamma}(1+\crg)^2
  \end{pmatrix},\quad
\mathbf{0}=(0,0)
\]
and $\{\alpha,\beta,\gamma\}=I$.
Using Lemma~\ref{square},  
we conclude 
\begin{align*}
 W(\theta_0, \theta ) 
&=4r^2\sin^2\frac12(\theta-\theta_0)
-\frac{r^4}{a_r(\theta ,\theta _0)}
  \frac{\lambda_\alpha \lambda_\beta \sin^2(\theta -\theta _0)}
   {(\lambda_\alpha+\lambda_\beta)
   (\lambda_\alpha+\lambda_\gamma)} +o(\theta^2),
\end{align*}
where
\[
 a_r(\theta ,\theta _0)
=\lambda_\alpha(1+\ccag)(1+\crg)
+r^2\sin\theta_0\sin\theta+\lambda_\beta.
\]
Since the limit of $a_r(\theta_0,\theta )$ exists 
as $\theta\rightarrow \theta_0$,
we have
\[
 \lim_{\theta\rightarrow \theta_0}
 \frac{W(\theta_0,\theta)}{(\theta - \theta_0)^2} 
=r^2 -\frac{r^4}{a_r(\theta_0 ,\theta _0)} 
 \frac{\lambda_\alpha \lambda_\beta}
  {(\lambda_\alpha+\lambda_\beta)(\lambda_\alpha+\lambda_\gamma)}.
\]
It follows that
\begin{align*}
 L(r)=\int_0^{2\pi} 
        r \left(1-\frac{1}{2}\frac{r^2}{a_r(\theta ,\theta)} 
          \frac{\lambda_\alpha \lambda_\beta}
             {(\lambda_\alpha+\lambda_\beta)(\lambda_\alpha+\lambda_\gamma)} 
      +o(r^2)\right)  d \theta. 
\end{align*}
Because $a_0(\theta,\theta)=(\lambda_\alpha+\lambda_\beta)$, 
using Lemma~\ref{taylor} and the bounded convergence theorem, 
we obtain 
\[
 K(A,B)=\frac{3\lambda_\alpha \lambda_\beta}
             {(\lambda_\alpha+\lambda_\beta)^2
              (\lambda_\alpha+\lambda_\gamma)}.
\]
We can prove the case of  $i\neq1$ and $j=d$ 
in a similar way. 
\begin{flushleft}
\textbf{Proof of (4),(6) and (9)}  
\end{flushleft}
We take $(A,B)$ in \eqref{cov} as    
$(e_+,f_{kl})$  ($\{1,d\}\cup\{k,l\}=\emptyset$),  
$(e_{i j},f_{kl})$  ($\{i,j\}\cup\{k,l\}=\emptyset$) and 
$(f_{i j},f_{kl})$ ($\{i,j\}\cup\{k,l\}=\emptyset$) 
in this order.
Moreover we set 
$I=\{1,d\}$ in the case (4) and 
$I=\{i,j\}$ in the case of (6) and (9).
We notice that 
for any $\alpha,\beta \notin I$,
$(\alpha,\beta)$-components of $X$ 
are independent of the variables $r$ and $\theta$.
If we set
\begin{align*}
  \widetilde{X}_c(\theta)=
  \begin{pmatrix}
     x_{\alpha\alpha} &x_{\alpha\beta} \\
     x_{\beta\alpha} &x_{\beta\beta}
 \end{pmatrix},\ 
 \widetilde{X}_s(\theta)=
   \begin{pmatrix}
     x_{k k} &x_{k l} \\
     x_{l k} &x_{l l}
   \end{pmatrix},
\end{align*}
we obtain 
\begin{align}\label{saigo}
 W(\theta_0,\theta)
=&\tr \widetilde{X}_c(\theta_0)+\tr \widetilde{X}_c(\theta)
   -2\tr\left(\widetilde{X}_c(\theta_0)^{\frac{1}{2}}
             \widetilde{X}_c(\theta)
            \widetilde{X}_c(\theta_0)^{\frac{1}{2}}\right)^{\frac12}\\ 
&+\tr \widetilde{X}_s(\theta_0)+\tr \widetilde{X}_s(\theta) \notag
   -2\tr\left(\widetilde{X}_s(\theta_0)^{\frac{1}{2}}
             \widetilde{X}_s(\theta)
            \widetilde{X}_s(\theta_0)^{\frac{1}{2}}\right)^{\frac12},
\end{align}
where 
$\{\alpha,\beta\}=I$.
Using Lemma~\ref{square}, we conclude 
\[
 \lim_{\theta\rightarrow \theta_0}
 \frac{W(\theta_0,\theta)}{(\theta - \theta_0)^2}
 =\left(\lim_{\theta\rightarrow \theta_0}
 \frac{r\sin (\theta - \theta_0)}{\theta - \theta_0}\right)^2 \
=r^2.
\]
It follows that $L(r)=2 \pi r$ and 
$K(A,B)=0$.
\begin{flushleft}
\textbf{Proof of (10)}  
\end{flushleft}
Without loss of generality, 
we may assume $j<k$.
Taking  $A$ and $B$ as 
$f_{i j}$ 
and $f_{i k}$ in \eqref{cov} respectively.
We  notice that 
for any $\alpha,\beta \notin \{i,j,k\}$,
$(\alpha,\beta)$-components of $X$ 
are independent of the variables $r$ and $\theta$.
If we set
\begin{align*}
 &\widetilde{X}(\theta)=
  \begin{pmatrix}
     x_{ii} &x_{i j} &x_{i k} \\
     x_{j i} &x_{j j} &x_{j k} \\
     x_{k i} &x_{k j} &x_{k k}
  \end{pmatrix}\\
 &=\begin{pmatrix}
    \lambda_i+\lambda_j\cij^2+\lambda_k\sik^2 &(\LIJ)\cij & (\IK)  \sik \\
    (\LIJ) \cij & \lambda_j+\lambda_i\cij^2 & \lambda_i\cij\sik \\
     (\IK) \sik & \lambda_i \cij\sik &\lambda_k+\lambda_i\sik^2
   \end{pmatrix},
\end{align*}
we obtain 
\begin{align}\label{sixth}
W(\theta_0,\theta) 
& =\tr \widetilde{X}(\theta_0)+\tr \widetilde{X}(\theta)
   -2\tr\left(\widetilde{X}(\theta_0)^{\frac{1}{2}}
             \widetilde{X}(\theta)
            \widetilde{X}(\theta_0)^{\frac{1}{2}}\right)^\frac12.
\end{align}
For the value of the last term in \eqref{sixth},
Lemma~\ref{square}  can not be used 
as the size of matrices is $3 \times 3$.

We define some notations:
\begin{align*}
 &A=A_{\theta_0}(\theta)
  =\widetilde{X}(\theta_0)^{\frac{1}{2}}
    \widetilde{X}(\theta)
    \widetilde{X}(\theta_0)^{\frac{1}{2}} \\
 &B=B_{\theta_0}(\theta)
  =\left(\widetilde{X}(\theta_0)^{\frac{1}{2}}
          \widetilde{X}(\theta)
          \widetilde{X}(\theta_0)^{\frac{1}{2}}\right)^{\frac12} \\
&\{\sigma_\alpha =\sigma_{\theta_0} (\theta)_\alpha\}_{\alpha=1}^{3}
  :\textrm{eigenvalues of }B \\
 &f_{\theta_0}(\theta)
 =\tr B =\sigma_1+\sigma_2+\sigma_3  \\
 &g_{\theta_0}(\theta)
 =\tr A =\sigma_1^2+\sigma_2^2+\sigma^2_3  \\
 &h_{\theta_0}(\theta)
 =\sigma_1\sigma_2+\sigma_2\sigma_3+\sigma_3\sigma_1  \\
 &\varphi_{\theta_0}(\theta)
 =\sigma_1^2\sigma_2^2+\sigma_2^2\sigma_3^2+\sigma_3^2\sigma_1^2  \\
 &D_{\theta_0}(\theta)
 =\det B=(\det A)^{\frac12}=\sigma_1\sigma_2\sigma_3 
\end{align*}
Rewriting \eqref{sixth} with the Taylor approximation 
of $f_{\theta_0}(\cdot)$ at $\theta_0$, 
we obtain
\begin{align*}
 W(\theta_0,\theta)
&=-2f^{\prime}_{\theta_0}(\theta_0)(\theta-\theta_0)
  -f^{\prime \prime}_{\theta_0}(\theta_0)(\theta-\theta_0)^2
  +o(|\theta-\theta_0|^2).
\end{align*}
Since we can get the values of $g$, $\varphi$ and $D$
without information of $X^{1/2}$,
we compute $f^{\prime}$ and $f^{\prime\prime}$
by using these values.

We calculate $f^{\prime}_{\theta_0}(\theta_0)$ first.
Differentiating 
$ B_{\theta_0}(\theta)\cdot B_{\theta_0}(\theta) =A_{\theta_0}(\theta)$
with respect to $\theta$,
we have 
\[
 B^{\prime}_{\theta_0}(\theta)   B_{\theta_0}(\theta) 
 + B_{\theta_0}(\theta)   B^{\prime}_{\theta_0}(\theta) 
 =A^{\prime}_{\theta_0}(\theta)
 =\widetilde{X}(\theta_0)^{\frac{1}{2}}
    \widetilde{X}^{\prime}(\theta)
    \widetilde{X}(\theta_0)^{\frac{1}{2}}.
\]
After multiplying  $B_{\theta_0}(\theta)^{-1}$ from the left, 
taking the trace gives
\begin{align*}
 \tr  B^{\prime}_{\theta_0}(\theta_0)  
 +\tr (B_{\theta_0}(\theta_0)   B^{\prime}_{\theta_0}(\theta_0)  
     B_{\theta_0}(\theta_0)^{-1})
 &=2f^{\prime}_{\theta_0}(\theta_0)
\end{align*}
at $\theta=\theta_0$.
Because $\tr \widetilde{X}(\theta)$ is constant, 
at $\theta=\theta_0$ the right hand side is equal to
\begin{align*}
 \tr (\widetilde{X}(\theta_0)^{\frac{1}{2}}
    \widetilde{X}^{\prime}(\theta_0)
    \widetilde{X}(\theta_0)^{\frac{1}{2}}
    \widetilde{X}(\theta_0)^{-1})
 =\tr \widetilde{X}^{\prime}(\theta_0) 
 =\left(\tr \widetilde{X}(\theta)\right)^{\prime}
   \bigg|_{\theta =\theta_0} =0.
\end{align*}
Therefore we conclude 
\begin{equation}\label{ikkai}
 f^{\prime}_{\theta_0}(\theta_0)=0.
\end{equation}

Next we compute $f^{\prime \prime}_{\theta_0}(\theta_0)$.
Differentiating 
 $f^2=g+2h$ at $\theta=\theta_0$,  
we have
\[
  2f_{\theta_0}(\theta_0)f^{\prime}_{\theta_0}(\theta_0)
 =g^{\prime}_{\theta_0}(\theta_0)+2h^{\prime}_{\theta_0}(\theta_0),
\]
proving 
\[
 2h^{\prime}_{\theta_0}(\theta_0)=-g^{\prime}_{\theta_0}(\theta_0)
\]
because of \eqref{ikkai}. 
Differentiating once more,
\[
 f^{\prime\prime}_{\theta_0}(\theta)
 =-\frac{f^{\prime}_{\theta_0}(\theta)}
        {2f_{\theta_0}(\theta)^2}
   \left(
        g^{\prime}_{\theta_0}(\theta)
        +2h^{\prime}_{\theta_0}(\theta)
   \right)
  +\frac{g^{\prime \prime}_{\theta_0}(\theta)+2h^{\prime \prime}_{\theta_0}(\theta)}{2f_{\theta_0}(\theta)}.
\]
Because of \eqref{ikkai}, 
we get at $\theta=\theta_0$
\begin{gather}\label{fsecond} 
  f^{\prime\prime}_{\theta_0}(\theta_0)
 =\frac{g^{\prime \prime}_{\theta_0}(\theta_0)
        +2h^{\prime \prime}_{\theta_0}(\theta_0)}
        {2f_{\theta_0}(\theta_0)}.
\end{gather}
We compute directly
\begin{align*}
 g_{\theta_0}(\theta)
 =\sum_{\alpha, \beta \in \{i,j,k\}} x_{\alpha\beta}(\theta_0)x_{\beta\alpha}(\theta).
\end{align*}
This enables us to get the derivatives of $ g_{\theta_0}(\theta)$.
Because $B_{\theta_0}(\theta_0)=X(\theta_0)$, 
using the relation 
\[
 \det(t E-B)=t^3-t^2\cdot f+t \cdot h-D,
\]
we have
\begin{align*}
 h_{\theta_0}(\theta_0)
&=\sum_{\begin{array}{c}
        \alpha,\beta \in \{i,j,k\} \atop \alpha \neq \beta 
\end{array}}
  \left(x_{\alpha\alpha}(\theta_0)x_{\beta\beta}(\theta_0)
       -x_{\alpha\beta}(\theta_0)^2\right).
\end{align*}
While it is hard to compute $B_{\theta_0}(\theta)$ directly,
it is also hard to 
know the value of $h_{\theta_0}(\theta)$.
We want to derive $h^{\prime\prime}_{\theta_0}(\theta)$ 
without the information of $B_{\theta_0}(\theta)$.
So differentiating  
 $h^2=\varphi+2D f$
twice,
we have 
\begin{align*}
  2(h^{\prime}_{\theta_0}(\theta))^2+
 2h_{\theta_0}(\theta) h^{\prime\prime}_{\theta_0}(\theta) 
 =\varphi^{\prime\prime}_{\theta_0}(\theta)+
   4D^{\prime}_{\theta_0}(\theta)f^{\prime}_{\theta_0}(\theta)
   +2D^{\prime\prime}_{\theta_0}(\theta)f_{\theta_0}(\theta)
   +2D_{\theta_0}(\theta)f^{\prime\prime}_{\theta_0}(\theta).
\end{align*}
At $\theta=\theta_0$, 
we have 
\begin{equation}\label{hsecond}
 h^{\prime\prime}_{\theta_0}(\theta_0) 
 =-\frac{g^{\prime}_{\theta_0}(\theta_0)^2}{4h_{\theta_0}(\theta_0)}
  +\frac{
 \varphi^{\prime\prime}_{\theta_0}(\theta_0)
      +2D^{\prime\prime}_{\theta_0}(\theta_0)f_{\theta_0}(\theta_0)
      +2D_{\theta_0}(\theta_0)f^{\prime\prime}_{\theta_0}(\theta_0)
           }{2h_{\theta_0}(\theta_0)}.
\end{equation}
In order to analyze \eqref{hsecond},
we consider $D_{\theta_0}(\theta)$ and $\varphi_{\theta_0}(\theta)$.
From the definition, we can compute $D_{\theta_0}(\theta)$ directly as 
\begin{align*}
D_{\theta_0}(\theta) 
=\lambda_i\lambda_j\lambda_k 
  \left[1-(\ccij^2+\ssik^2)\right]
  \left[1-(\cij^2+\sik^2)\right]. 
\end{align*}
We next consider  $\varphi_{\theta_0}(\theta)$.
Using the equation 
\begin{align*}
\det (t E-A)
=\det \widetilde{X}(\theta_0) \cdot
  \det (t\widetilde{X}(\theta_0)^{-1}-\widetilde{X}(\theta)),
\end{align*}
and the relation
\begin{equation*}
 \det(t E-A)=t^3-t^2\cdot g(\theta)+t \cdot \varphi-D^2 ,
\end{equation*}
we conclude
\begin{align}\label{ppp}
 \varphi_{\theta_0}(\theta)
&=\det (\widetilde{X}(\theta_0) \widetilde{X}(\theta))
  \cdot \tr (Y(\theta_0)Y(\theta)),\ 
\text{where }
Y(\theta)=\widetilde{X}(\theta)^{-1}.
\end{align}
Since \eqref{ppp} depends only on $\widetilde{X}(\theta)$,
we can obtain the value of $\varphi_{\theta_0}(\theta)$.
Therefore we can now specify the value of 
$h^{\prime\prime}_{\theta_0}(\theta_0)$ in\eqref{hsecond}.

Inserting \eqref{hsecond} into \eqref{fsecond},
we obtain
\[
 W(\theta_0,\theta)
=-f^{\prime\prime}_{\theta_0}(\theta_0)+o(|\theta-\theta_0|^2)
=-\frac{\beta_r(\theta_0)}{\alpha_r(\theta_0)}+o(|\theta-\theta_0|^2),
\]
where
\begin{align}
 \label{bunbo}
 \alpha_r(\theta_0)
 =&2\left[f_{\theta_0}(\theta_0)h_{\theta_0}(\theta_0)
  -D_{\theta_0}(\theta_0) \right]\\ \notag
 =&2\IJK\\ \notag
 &+r^2 \left[\lambda_j^2+\lambda_k^2
       +4\lambda_j\lambda_k
       +\lambda_i\lambda_j
       +\lambda_i\lambda_k 
+(\lambda_j-\lambda_k)(\lambda_j+\lambda_k+3\lambda_i)
       \cos \theta_0 
 \right],\\ 
 \label{bunshi}
 \beta_r(\theta_0)
 =&h_{\theta_0}(\theta_0)g^{\prime\prime}_{\theta_0}(\theta_0)
 -\frac12g^{\prime}_{\theta_0}(\theta_0)^2
 +\varphi^{\prime\prime}_{\theta_0}(\theta_0)
 +2D^{\prime\prime}_{\theta_0}(\theta_0) \\ \notag
 =&-2r^2\IJK \\ \notag
 &-r^4\left[(\JK)(\LIJ+\lambda_k)
       +( \lambda_j-\lambda_k ) ( \lambda_j+\lambda_k+3\lambda_i )
       \cos \theta_0 
 \right]. 
\end{align}
Therefore we have
\[
 W(\theta_0)
=\lim_{\theta \to \theta_0}
 \frac{W(\theta,\theta_0)}{\theta^2}
=-\frac{\beta_r(\theta_0)}{\alpha_r(\theta_0)}. 
\]
If we set 
\begin{align*}
&L=2\IJK,\\
&a=\lambda_j^2+\lambda_k^2
       +4\lambda_j\lambda_k
       +\lambda_i\lambda_j
       +\lambda_i\lambda_k
       +(\lambda_j-\lambda_k)(\lambda_j+\lambda_k+3\lambda_i)
       \cos \theta, \\
&b=(\JK)(\LIJ+\lambda_k)
       +(\lambda_j-\lambda_k)(\lambda_j+\lambda_k+3\lambda_i)
       \cos \theta,
\end{align*}
we have
\begin{align*}
 L(r)
=\int_0^{2 \pi} r 
  \left(1
      +\frac{r^2(b-a)}{2(L+r^2a)}
      +o(r^2)\right) d\theta.
\end{align*}
Using Lemma~\ref{taylor} and the bounded convergence theorem, 
we obtain 
\begin{align*}
 2 \pi \frac{K(u,v)}{6}
= \int_0^{2\pi} \lim_{r\searrow 0}
        \frac{a-b}{2(L+r^2a)}
      d \theta 
= 2\pi  \frac{a-b}{2L},     
\end{align*}
which implies that 
\[
 K(f_{i j},f_{i k})=
\frac{3\lambda_k\lambda_j}{\IJK}.
\]

This completes the proof of Theorem~\ref{curv}. 

\section{Remarks to Theorem~\ref{curv}}
In this section
we consider the case $d=2$ in particular.
\subsection{Geometric interpolations of Theorem~\ref{curv}}
\begin{lemma} \label{rotation}
Any $V =(v_{i j })\in \SSym$ is diagonalized 
by some special orthogonal matrix.
In other word, there exists some 
$\theta \in \R$ such that the rotation matrix 
\[
 R(\theta)=
 \begin{pmatrix}
 \cos \theta   & -\sin \theta \\
 \sin \theta  &  \cos \theta\\
 \end{pmatrix}
\]
 diagonalizes $V$. 
\end{lemma}
\begin{proof}
In the case of $v_{12}=0$, 
we set $\theta=0$.
While in the case of $v_{12}\neq 0$, 
since
\begin{align*}
 \ten R(\theta)V R(\theta) 
=
 \begin{pmatrix}
  v_{11}\cos^2 \theta +v_{12}\sin 2 \theta +  v_{22}\sin^2 \theta
& v_{12}\cos 2 \theta  +2^{-1}(v_{11}-v_{22})\sin2 \theta \\
  v_{12}\cos 2 \theta  +2^{-1}(v_{11}-v_{22})\sin2 \theta 
& v_{22}\cos^2 \theta -v_{12}\sin 2 \theta +  v_{11}\sin^2 \theta
 \end{pmatrix},
\end{align*}
$ \ten R(\theta)V R(\theta)$ is a diagonal matrix if and only if 
\begin{equation}\label{cot}
 v_{12}\cos 2 \theta =-2^{-1}(v_{11}-v_{22})\sin2 \theta
 \Leftrightarrow 
 \cot 2\theta = (v_{11}-v_{22})/2v_{12}.
\end{equation}
Because $\cot 2 \theta$ can take any value,  
\eqref{cot} always holds true.
\end{proof}
For $\alpha,\beta>0$,
we denote
\[
 (\alpha,\beta;\theta)=N\left(R(\theta) 
 \begin{pmatrix}
 \alpha ^2& 0 \\
        0 & \beta ^2 \\
 \end{pmatrix}
 \ten R(\theta)\right).\]
We abbreviate $\N^2_0(R(\theta))$ 
as $\N^2_0(\theta)$.
We also set 
 $\Lambda=
 \{(\lambda,\lambda;\theta) \ |\ \lambda >0 \ \}.$
Since $(\alpha,\beta;\pi/2+\theta)=(\beta,\alpha;\theta)$,
the expression $(\alpha,\beta;\theta)$ is not 
a global coordinate system.
Even if we consider under modulo $\pi/2$, 
there is no uniqueness of diagonalizing matrix 
if $\alpha$ is equal to $\beta$. 

Throughout this section, 
we fix $\rho=(\alpha,\beta;0)$.
We regard Gaussian measures $(\alpha,\beta;\theta)$ as ellipsoids:
$(\alpha,\beta)$ specifies the length of the axes  with 
the angle $\theta$ of major and minor axes.
Let $X,Y$ and $Z$ be matrices defined by
\begin{equation}\label{xyz}
X=e_{11}=\frac{1}{\gamma} 
          \begin{pmatrix}
                  1& 0 \\
                   0&-1\\
           \end{pmatrix},\ 
Y=f_{12}=\frac{1}{\gamma}
          \begin{pmatrix}
                  0& -1 \\
                 -1&  0\\
           \end{pmatrix},\ 
Z=e_+=\frac{1}{\gamma} 
          \begin{pmatrix}
                  1& 0 \\
                  0& 1\\
           \end{pmatrix}
\end{equation}
where $\gamma=(\alpha ^2 +\beta ^2)^{1/2}$.
Using this expression, we get 
\begin{align*}
&\exp _{\rho} r X
=N(U),
&U&=\frac{1}{\gamma ^2} 
     \begin{pmatrix}
         \alpha ^2 (\gamma +r)^2& 0 \\
         0&\beta ^2 (\gamma -r)^2  \\
       \end{pmatrix}, \\
&\exp _{\rho} r Y
=N(V),
&V&=\frac{1}{\gamma ^2} 
     \begin{pmatrix}
       \alpha ^2 \gamma ^2 + \beta ^2 r^2 & \gamma ^3 r \\
       \gamma ^3 r &  \alpha ^2 r^2 + \beta ^2 \gamma ^2 \\
       \end{pmatrix} ,\\
&\exp _{\rho} r Z
=N(W),
&W&=\frac{1}{\gamma ^2} 
     \begin{pmatrix}
         \alpha ^2 (\gamma +r)^2& 0 \\
         0&\beta ^2 (\gamma +r)^2  \\
       \end{pmatrix}. 
\end{align*}

%
We notice that $Y$ changes the axial angle of the ellipsoid,
while $X$ and $Z$ do not, see Figure~1.

\quad

\quad
 
\begin{center}
\input{fig.tex}
\end{center}

\quad

\quad

Consequently, $X,Z \in T_{\rho}\N^2_0(0)$ and $K(X,Z)=0$ 
by Corollary~\ref{flat}. 
For changing the axial angle of ellipsoid,
the sectional curvature can not vanish.
\subsection{Correspondence to other results}
First, we consider the correspondence to the result of Otto~\cite{Ot}.
He obtained an explicit expression 
of sectional curvatures of $\ppac$ formally. 
By making this method rigorous, 
we give an explicit expression 
of sectional curvatures of $\N^d$ in~\cite{takatsu}.
He introduced a manifold $\mm$ 
which consists of all diffeomorphisms of $\R ^d$ 
and an isometric submersion from $\mm$ into $\ppac$
(he also sloppied about a differential structure of $\mm$.)
He defined a metric $g^*$ on $\mm$ 
which carried the geometry of the $L^2$-space.
Therefore, ($\mm$,$g^*$) is flat.
Using O'Neill's  formula~\cite{oni}, 
he showed the sectional curvatures of $\ppac$ is given by
\begin{align*} 
K(\psi _1,\psi _2)\det\left(g_\rho (\psi _i,\psi _j)\right)
=\frac{3}{4}\int _{\R ^d}\rho |u|^2 \geq 0,
\end{align*}
where $\rho \in \ppac$ and 
$\psi _1$, $\psi _2$, $\psi$ are tangent vectors at $\rho$
 given by 
\[
u=\nabla \psi -[\nabla\psi _1,\nabla \psi _2]
\quad \text{and} \quad
\mathrm{div}(\rho (\nabla \psi-[\nabla\psi _1,\nabla \psi _2]))=0.
\]
This guarantees that $\ppac$ is a space of non-negative curvature.
In addition, $K(\psi_1,\psi _2)=0$ if and only if
$\mathrm{Hess} \psi _1$ and $\mathrm{Hess} \psi _2$ 
pointwise commute.
For $X$, $Y$ and $Z$ in \eqref{xyz},
let $\X$, $\Y$ and $\Z$ be functions whose gradients 
are $X$, $Y$ and $Z$, respectively. 
Since $Z$ is pointwise commutative with the Hessian of any functions,
$K(\X,\Z)=K(\Y,\Z)=0$ follows. 
In the case of $\X,\Y$, we demonstrate 
that Theorem~\ref{curv} coincides with Otto's result.

Let $\rho_0$ be the standard Gaussian measure on $\R^2$, 
that is $\rho_0=(1,1;\theta)$.
Moreover, we define a Gaussian measure $\rho$ and
a diffeomorphism $\Psi$ respectively as follows:
\begin{align*}
\rho=(\alpha,\beta;\theta),\quad
\Psi(x)=\frac{1}{\gamma}R 
          \begin{pmatrix}
          \alpha^2 & 0\\
               0   &\beta^2\\
          \end{pmatrix} 
 \ten Rx,
\end{align*}
where $\gamma=(\alpha^2+\beta^2)^{1/2}$ and $R=R(\theta)$.
Then, the submersion sends $\Psi$ to $\rho$.
We choose tangent vectors $\psi_1,\psi_2$ at $\rho$ 
as $\X$ and $\Y$, 
corresponding to $X$ and $Y$ respectively.
In terms of Otto's result, 
we conclude that
\begin{align*}
&g_{\rho}(\psi_i,\psi_j)
 =\delta_{i j} \quad (i,j=1,2) , \\
&[\nabla \psi _1,\nabla \psi _2](x)
 =\frac{2}{\gamma ^2} R 
  \begin{pmatrix}
      0&1\\
      1&0\\
  \end{pmatrix} \ten Rx,\\
&\psi(x)
=\frac{1}{\gamma ^4} \ten x R 
   \begin{pmatrix}
    0&\alpha^2 -\beta ^2\\
  \alpha ^2-\beta ^2&0\\
   \end{pmatrix} \ten Rx, \\
&u(x)= \frac{4}{\gamma ^4} R 
         \begin{pmatrix}
               0&  \alpha ^2 \\
        -\beta ^2 &0 \\
        \end{pmatrix} \ten Rx.                              
\end{align*}
Finally, we obtain
\begin{align*}
K(\psi _1,\psi _2) 
&=\frac{3}{4\det(g_{\rho}(\psi _1,\psi _2))} 
                   \int _{\R ^2} |u(x)|^2 \rho(x) d x
=\frac{12\alpha ^2 \beta ^2 }{(\alpha^2 +\beta ^2)  ^3}.
\end{align*}
Thus we confirm the equivalence between 
Theorem \ref{curv} and Otto's result.
In \cite{takatsu}, 
the sectional curvature of $\N^d$ was also obtained 
using Riemannian submersion.
\subsection{$\N^d_0$ as Alexandrov spaces.}
Next, we consider the correspondence to results 
when we regard $\ppac$ and $\N^d$ as Alexandrov spaces.
Details can be found in ~\cite{takatsu}.

It is well-known that
$L^2$-Wasserstein space over 
an Alexandrov space of non-negative curvature
is also an Alexandrov space of non-negative curvature
(see~\cite[Proposition~2.10]{st}.)
Therefore $(\mathcal{P}_2(\R^d),W_2)$ is  
an  Alexandrov space of non-negative curvature.
Since $\N^d_0 \subset \pr$ 
is a geodesically convex subset,
$(\N^d_0,W_2)$ is also an Alexandrov space 
of non-negative curvature
(But it is not complete,
the completion of $\N^d$ is given in~\cite{takatsu}.)

Lott and Villani~\cite{lv} made Otto's results rigorous 
by looking at the space of probability measures as an Alexandrov space,
They treated the space of probability measures $\ppm$ over 
a smooth compact connected manifold $M$, 
and proved  that $M$ has non-negative sectional 
curvature if and only if $\ppm$ has non-negative Alexandrov curvature 
(\cite[Theorem A.2]{lv}.)
They moreover defined the angle between 
the geodesics in $\pmac$ (\cite[Theorem A.17]{lv}.)
We demonstrate it in the case of $M$ as $\R ^d $, 
while $\R^d$ is not a compact manifold. 
Fix $\rho \in \ppac$.
If $\phi$ is a function on $\R^d$ so that 
$\phi+|\cdot|^2/2$ is convex,
then  it is regarded as a tangent vector of $\ppac$ at $\rho$.
Let $\phi$ and $\psi$ be such functions.
If we set
\begin{align*}
\mu(t)=[\id+t\phi]_{\sharp}\rho,\ \nu(t)=[\id+t\psi]_{\sharp}\rho
\end{align*}
for $t\in[0,1]$,
then $\mu(t)$ and $\nu(t)$ are geodesics starting at $\rho$.
The angle between $\mu(t)$ and $\nu(t)$ is given by
\begin{align*}
\cos \measuredangle (\mu, \nu)
&=\frac{\int_{\R^d} \langle \nabla \phi(x),\nabla \psi (x)\rangle d \rho (x)}
       {\sqrt{\int _{\R^d}|\nabla \phi (x)|^2 d \rho (x)}
       \sqrt{\int _{\R^d}|\nabla \psi  (x)|^2 d \rho (x)}}.
\end{align*}
Corresponding to this, 
we can measure the angle between 
$\N^d_0(\theta)$ and $\N^d_0(\varphi)$. 
\begin{prop} \label{angleprop}
For $\theta$ and $\varphi \in \left(-\pi/4,\pi/4 \right] $, 
the angle between $\N^d_0(\theta)$ and $\N^d_0(\varphi)$ 
is $2|\theta -\varphi|$.
\end{prop}
\begin{proof}
Let $\theta$ and $\varphi$ be as above.
For any $\alpha,\beta ,\lambda >0$,  
\begin{align*}
W_2((\alpha,\beta;\theta),(\lambda,\lambda;\theta)) ^2
&=2\left(\lambda - \frac{\alpha + \beta}{2}\right) ^2 +\frac12 (\alpha -\beta )^2.
\end{align*}
Hence, the distance 
from $(\alpha,\beta;\theta)$ to $\Lambda$ 
is $(\alpha -\beta )/\sqrt2$ and 
the image of the nearest point projection 
is 
\[
 \rho=\left(\frac{\alpha + \beta}{2},
            \frac{\alpha + \beta}{2};\theta\right).
\]
Let $X(\theta)$ and $X(\varphi)$ be 
symmetric matrices given by
\begin{align*}
 & X(\theta)
=R(\theta)
\begin{pmatrix}
 (\alpha+\beta)^{-1}(\alpha - \beta) & 0 \\
 0 & (\alpha+\beta)^{-1}(\alpha - \beta) \\
 \end{pmatrix}
 R(-\theta) \\
& X(\varphi)
=R(\varphi)
\begin{pmatrix}
 (\alpha+\beta)^{-1}(\alpha - \beta) & 0 \\
 0 & (\alpha+\beta)^{-1}(\alpha - \beta) \\
 \end{pmatrix}
 R(-\varphi),
\end{align*}
then we get
\[
\exp_{\rho} X(\theta) =(\alpha ,\beta;\theta),
\quad
\exp _{\rho} X(\varphi) =(\alpha ,\beta;\varphi).
\]
Since 
\begin{align*}
g_{\rho}(X(\theta),X(\varphi))
=\frac{(\alpha+\beta)^2}{4} \tr X(\theta)X(\varphi)
=\frac12 (\alpha -\beta)^2\cos 2 (\theta -\varphi ) 
\end{align*}
and 
\[
 g_{\rho}(X(\theta),X(\theta))
=g_{\rho}(X(\varphi),X(\varphi))
=\frac12(\alpha -\beta )^2, 
\]
we have 
\begin{align} \label{angle}
\frac{g_{\rho}(X(\theta),X(\varphi))}
{\sqrt{g_{\rho}(X(\theta),X(\theta))g_{\rho}(X(\varphi),X(\varphi))}}
=\cos 2(\theta -\varphi). 
\end{align}
Therefore we conclude
\begin{align*}
\measuredangle (\N^d_0(\theta),\N^d_0(\varphi))
&=\mathrm{Arccos} \frac{g_{\rho}(X(\theta),X(\varphi))}
{\sqrt{g_{\rho}(X(\theta),X(\theta))g_{\rho}(X(\varphi),X(\varphi))}}
= 2|\theta -\varphi|. 
\end{align*}
\end{proof}



\begin{thebibliography}{10}

\bibitem{amari}
S.~Amari, \emph{Differential-geometrical methods in statistics}, Lecture Notes
  in Statistics, vol.~28, Springer-Verlag, New York, 1985.

\bibitem{bre}
Y.~Brenier, \emph{Polar factorization and monotone rearrangement of
  vector-valued functions}, Comm. Pure Appl. Math. \textbf{44} (1991), no.~4,
  375--417.

\bibitem{CMV}
J.A. Carrillo, R.J. McCann, and C.~Villani, \emph{Contractions in the
  2-{W}asserstein length space and thermalization of granular media}, Arch.
  Ration. Mech. Anal. \textbf{179} (2006), no.~2, 217--263.

\bibitem{DL}
D.C. Dowson and B.V. Landau, \emph{The {F}r\'echet distance between
  multivariate normal distributions}, J. Multivariate Anal. \textbf{12} (1982),
  no.~3, 450--455.

\bibitem{GHL}
S.~Gallot, D.~Hulin, and J.~Lafontaine, \emph{Riemannian geometry}, third ed.,
  Universitext, Springer-Verlag, Berlin, 2004.

\bibitem{MG}
W.~Gangbo and R.J. McCann, \emph{The geometry of optimal transportation}, Acta
  Math. \textbf{177} (1996), no.~2, 113--161.

\bibitem{GS}
C.R. Givens and R.M. Shortt, \emph{A class of {W}asserstein metrics for
  probability distributions}, Michigan Math. J. \textbf{31} (1984), no.~2,
  231--240.

\bibitem{lv}
J.~Lott and C.~Villani, \emph{Ricci curvature for metric-measure spaces via
  optimal transport compact alexandrov spacess}, Annals of Mathmatics (to
  appear).

\bibitem{Mc}
R.J. McCann, \emph{A convexity principle for interacting gases}, Adv. Math.
  \textbf{128} (1997), no.~1, 153--179.

\bibitem{KS}
M.Knott and C.S. Smith, \emph{On the optimal mapping of distributions}, J.
  Optim. Theory Appl. \textbf{43} (1984), no.~1, 39--49.

\bibitem{OP}
I.~Olkin and F.Pukelsheim, \emph{The distance between two random vectors with
  given dispersion matrices}, Linear Algebra Appl. \textbf{48} (1982),
  257--263.

\bibitem{oni}
B.~O'Neill, \emph{The fundamental equations of a submersion}, Michigan Math. J.
  \textbf{13} (1966), 459--469.

\bibitem{Ot}
F.~Otto, \emph{The geometry of dissipative evolution equations: the porous
  medium equation}, Comm. Partial Differential Equations \textbf{26} (2001),
  no.~1-2, 101--174.

\bibitem{st}
K.-T. Sturm, \emph{On the geometry of metric measure spaces. {I}}, Acta Math.
  \textbf{196} (2006), no.~1, 65--131.

\bibitem{takatsu}
A.~Takatsu, \emph{{W}asserstein geometry of {G}aussian measures}, in
  preparation.

\bibitem{Vi}
C.~Villani, \emph{Optimal transport, new and old}, Grundlehren der
  mathematischen Wissenschaften, vol. 338, Springer, Berlin, 2008.

\end{thebibliography}
\clearpage

\def\cprime{$'$}
\providecommand{\bysame}{\leavevmode\hbox to3em{\hrulefill}\thinspace}
\providecommand{\MR}{\relax\ifhmode\unskip\space\fi MR }
\providecommand{\MRhref}[2]{%
  \href{http://www.ams.org/mathscinet-getitem?mr=#1}{#2}
}
\providecommand{\href}[2]{#2}

\affiliationone{
   Asuka Takatsu \\
  Mathematical Institute,  
  Tohoku University \\
  Sendai 980-8578, Japan
   \email{sa6m21@math.tohoku.ac.jp}}
\end{document}